\RequirePackage{etoolbox}
\csdef{input@path}{{style/}{graphics/}}
\documentclass[MSNbibl,number,citesort,seceqn,dvips]{arxbj}
\usepackage{upgreek}

\aid{0}
\volume{21}
\issue{1}
\pubyear{2015}
\firstpage{335}
\lastpage{359}
\doi{10.3150/13-BEJ569} 

\makeatletter

\newcommand{\ignore}[1]{}{}

\newcommand{\eq}{\eqref}
\newcommand{\eqref}[1]{(\ref{#1})}
\def\E{{\mathbb{E}}}
\def\Var{\operatorname{Var}}
\def\cov{\operatorname{Cov}}
\def\P{{\mathbb{P}}}

\newtheorem{theorem}{Theorem}[section]
\newtheorem{lem}[theorem]{Lemma}
\newtheorem{cor}[theorem]{Corollary}
\newtheorem{prop}[theorem]{Proposition}
\newremark{rem}[theorem]{Remark}

\makeatother

\begin{document}
\begin{frontmatter}

\title{On the error bound in a combinatorial central limit theorem}
\runtitle{Combinatorial central limit theorem}

\begin{aug}
\author[a]{\inits{L.H.Y.}\fnms{Louis H.Y.} \snm{Chen}\corref{}\thanksref{a}\ead[label=e1]{matchyl@nus.edu.sg}} \and
\author[b]{\inits{X.}\fnms{Xiao} \snm{Fang}\thanksref{b}\ead[label=e2]{stafx@nus.edu.sg}}
\address[a]{Department of Mathematics,
National University of Singapore,
10 Lower Kent Ridge Road,
Singapore 119076,
Republic of Singapore.
\printead{e1}}
\address[b]{Department of Statistics and Applied Probability,
National University of Singapore,
6 Science Drive 2,
Singapore 117546,
Republic of Singapore.
\printead{e2}}
\end{aug}

\received{\smonth{11} \syear{2011}}
\revised{\smonth{9} \syear{2013}}

%
\begin{abstract}
Let $\mathbb{X}=\{X_{ij}\dvt  1\le i,j \le n\}$ be an $n\times n$ array of
independent random
variables where $n\ge2$. Let $\pi$ be a uniform random permutation of
$\{1,2,\dots, n\}$,
independent of $\mathbb{X}$, and let $W=\sum_{i=1}^n X_{i\pi(i)}$.
Suppose $\mathbb{X}$ is standardized
so that $\E W=0, \Var(W)=1$. We prove that the Kolmogorov distance
between the distribution of $W$ and
the standard normal distribution is bounded by $451\sum_{i,j=1}^n \E
|X_{ij}|^3/n$. Our approach is by
Stein's method of exchangeable pairs and the use of a concentration inequality.
\end{abstract}

%

\end{frontmatter}
%
\section{Introduction and statement of the main result}\label{sec1}

Motivated by permutation tests in non-parametric statistics,
Wald and Wolfowitz \cite{WaWo44} proved a central limit theorem for
the combinatorial
statistics $\sum_{i=1}^n a_i b_{\pi(i)}$ where $\{a_i, b_j\dvt  i,j\in
[n]:=\{1,2,\dots,n\}\}$ are real numbers and $\pi$ is a uniform
random permutation of $[n]$. Their result was generalized to real
arrays $\{c_{ij}\dvt  i,j\in[n]\}$ by Hoeffding \cite{Ho51}.
Extension to random
arrays $\{X_{ij}\dvt  i,j\in[n]\}$ where the $X_{ij}$ are independent
random variables was considered by Ho and Chen \cite{HoCh78}.
Using the
concentration inequality approach in Stein's method, they proved a
bound on the Kolmogorov distance between the distribution of $\sum_{i=1}^n {X_{i\pi(i)}}$ and the normal distribution with the same mean
and variance.
The bound in Ho and Chen \cite{HoCh78} is optimal only when
$|X_{ij}|\le C$ for
some $C>0$. A third-moment bound for a combinatorial central limit
theorem for real arrays $\{c_{ij}\dvt  i,j\in[n]\}$ was obtained by
Bolthausen \cite{Bo84}, who used Stein's method and induction.
However, the absolute
constant in the bound in Bolthausen \cite{Bo84} is not
explicit. A bound with an
explicit constant for real arrays with $|c_{ij}|\le C$ was obtained by
Goldstein \cite{Go05} using Stein's method and zero-bias
coupling (see also
Chen, Goldstein and Shao \cite{ChGoSh11}).
Under the same setting as Ho and Chen \cite{HoCh78},
Neammanee and Suntornchost \cite{NeSu05} stated a
third-moment bound. They used the same Stein identity in Ho and Chen \cite{HoCh78}, which dates back to Chen \cite
{Ch75}, and the concentration
inequality approach. However, there is an error in the proof in Neammanee and Suntornchost \cite{NeSu05},
where the first equality and the second inequality on page 576 are
incorrect because of the dependence among $S(\tau), \Delta S$ and
$M(t)$. Although the bound obtained by Neammanee and Rerkruthairat \cite
{NeRe12} (see Theorem~1.1
on page 1591) simplifies to a third-moment bound plus a term of a
smaller order, the latter contains an undetermined constant. Besides,
the proof of Theorem~1.1 uses a result of
Neammanee and Rattanawong \cite{NeRa08} (see (23) on page 21), whose
correctness is in question.

In this paper, we give a different proof of the combinatorial central
limit theorem. Our result gives a third-moment bound with an explicit
constant under the setting of Ho and Chen \cite{HoCh78}. Our approach is by Stein's
method of exchangeable pairs and the use of a concentration inequality.
The use of an exchangeable pair simplifies the construction of a Stein
identity as compared to the construction in Ho and Chen \cite
{HoCh78}, Neammanee and Suntornchost \cite{NeSu05} and Neammanee and Rerkruthairat
\cite{NeRe12}.

Stein's method was introduced by Stein \cite{St72,St86} and has become a
popular tool in proving distributional approximation results because of
its power in handling dependence among random variables. We refer to
Barbour and Chen \cite{BaCh05} and Chen, Goldstein and Shao \cite
{ChGoSh11} for an introduction to Stein's
method. The notion of exchangeable pair was introduced by Stein \cite{St86},
and first used in Diaconis \cite{Di77}. The concentration
inequality approach
was also introduced by Stein (see Ho and Chen \cite{HoCh78})
and was developed by
Chen \cite{Ch86,Ch98} and Chen and Shao
\cite{ChSh01,ChSh04}. This approach provides a
smoothing technique and a way of obtaining third-moment bounds on the
Kolmogorov distance.

\ignore{
Exchangeability was exploited extensively in Stein \cite
{St86} to prove
normal approximation results. However, not like the Berry-Esseen
theorem, it naturally results in fourth-moment bounds on Kolmogorov
distances. To overcome this difficulty, we first control the
probability that some random variable takes values in a small interval,
then use such a concentration inequality to prove third-moment bounds
on Kolmogorov distances. This approach is called the concentration
inequality approach. The concentration inequality approach has been
used to prove a Berry-Esseen bound for sums of i.i.d. random variables
in Ho and Chen \cite{HoCh78}, and for sums of locally
dependent random variables
in Chen and Shao \cite{ChSh04}. Recently, this approach has
been extended to
multivariate setting in \citeauthor{ChFa11} \cite{ChFa11}. }

The following is our main result.

\begin{theorem}\label{t1}
Let $\mathbb{X}=\{X_{ij}\dvt  i,j\in[n]\}$ be an $n\times n$ array of
independent random variables where $n\geq2$, $\E X_{ij}=c_{ij}$, $\Var
{X_{ij}}=\sigma_{ij}^2\ge0$ and $\E|X_{ij}|^3<\infty$. Assume
%
\begin{equation}
\label{1.1} c_{i\cdot}=c_{\cdot j}=0,
\end{equation}
where $c_{i\cdot}=\sum_{j=1}^n c_{ij}/n$, $c_{\cdot j}=\sum_{i=1}^n
c_{ij}/n$.
Let $\pi$ be a uniform random permutation of $[n]$, independent of
$\mathbb{X}$, and let $W=\sum_{i=1}^n X_{i\pi(i)}$.
Then
\begin{enumerate}[(2)]
\item[(1)]
%
\begin{equation}
\label{1.2} \Var(W)=\frac{1}{n} \sum_{i,j=1}^n
\sigma_{ij}^2 +\frac{1}{n-1} \sum
_{i,j=1}^n c_{ij}^2,
\end{equation}
\item[(2)]
assuming $\Var(W)=1$,
we have
%
\begin{equation}
\label{t1-1} \sup_{z\in\mathbb{R}}\bigl|P(W\le z)-\Phi(z)\bigr|\le451\gamma,
\end{equation}
where $\Phi$ is the standard normal distribution function and
%
\begin{equation}
\label{t1-2} \gamma=\frac{1}{n}\sum_{i,j=1}^n
\E|X_{ij}|^3.
\end{equation}
\end{enumerate}
\end{theorem}

Specializing Theorem~\ref{t1} to the case where $\sigma_{ij}=0$, we have
%
\begin{equation}
\label{c1-2} \sup_{z\in\mathbb{R}}\bigl|P(W\le z)-\Phi(z)\bigr|\le\frac{451}{n}
\sum_{i,j=1}^n |c_{ij}|^3.
\end{equation}
Moreover, if $|c_{ij}|\le C$ for all $i,j\in[n]$, then
%
\begin{equation}
\label{c1-3} \sup_{z\in\mathbb{R}}\bigl|P(W\le z)-\Phi(z)\bigr|\le451C.
\end{equation}

\begin{rem}
The constant $451$ in \eq{c1-2}, and therefore that in \eq{c1-3}, can
be reduced by a modification of the proof of Theorem~\ref{t1} for the
special case where $\sigma_{ij}=0$. The error bound in \eq{c1-2} was
obtained by Bolthausen \cite{Bo84} except that the constant in
his bound is not
explicit. In the case where $|c_{ij}|\leq C$, the error bound in \eq
{c1-3} is of the same order as that in Goldstein \cite{Go05}
and in Theorem~6.1
of Chen, Goldstein and Shao \cite{ChGoSh11}, although the constant in
Theorem~6.1 of Chen, Goldstein and Shao \cite{ChGoSh11} is $16.3$, which
is smaller. The constant $16.3$ is not
directly comparable to the constant $451$ in \eq{t1-1} and \eq{c1-2}
as the bounds in \eq{t1-1} and \eq{c1-2} are of third moment type,
which may be smaller.
\end{rem}

\begin{rem}\label{c2}
Any $n\times n$ array of independent random variables $\{Y_{ij}\dvt  i,j\in
[n]\}$ with $\E Y_{ij}=\mu_{ij}, \Var{Y_{ij}}=\sigma_{ij}^2\ge0$
and $\E|Y_{ij}|^3<\infty$ can be reduced to the array in Theorem~\ref
{t1} satisfying \eq{1.1} by defining $c_{ij}=\mu_{ij}-\mu_{i \cdot
}-\mu_{\cdot j}+\mu_{\cdot\cdot}$ and $X_{ij}=Y_{ij}-\mu_{i \cdot
}-\mu_{\cdot j}+\mu_{\cdot\cdot}$.
\end{rem}

Theorem~\ref{t1} has the following corollary for simple random
sampling from independent random variables.

\begin{cor}\label{c3}
Let $\{Y_1,\dots, Y_n\}$ be independent random variables with $\E
Y_i=\mu_i$, $\Var(Y_i)=\sigma_i^2\ge0$. For a positive integer
$k\le n$, let $\{Y_{\xi_1},\dots,Y_{\xi_k}\}$ be uniformly chosen
from $\{Y_1,\dots, Y_n\}$ without replacement, and let $V=\sum_{i=1}^k Y_{\xi_i}$. Then we have
%
\begin{eqnarray}
\label{c3-1} &&\hspace*{-15pt}\sup_{z\in\mathbb{R}}\biggl|P\biggl(
\frac{V- k\bar{\mu}}{\sigma}\le z\biggr)-\Phi (z)\biggr|\nonumber
\\[-8pt]\\[-8pt]
&&\hspace*{-15pt}\quad \le\frac{451}{n\sigma^3} \Biggl[ k\sum_{i=1}^n
\E \biggl| \frac{k}{n}(Y_i-\mu_i) +\frac
{n-k}{n}
(Y_i-\bar{\mu}) \biggr|^3 +(n-k)\sum
_{i=1}^n \biggl| \frac{k}{n}(\mu_i-
\bar{\mu}) \biggr|^3 \Biggr], \nonumber%
\end{eqnarray}
where
\begin{eqnarray*}
\bar{\mu}=\sum_{i=1}^n
\mu_i/n, \qquad \sigma^2=\Var(V)= \frac{k}{n} \sum
_{i=1}^n \sigma_i^2
+\frac{k(n-k)}{n(n-1)} \sum_{i=1}^n (\mu
_i-\bar{\mu})^2 ,
\end{eqnarray*}
and $\Phi$ is the standard normal distribution function.
\end{cor}

\begin{pf}
Using the same idea as in the Poisson approximation for the
hypergeometric distribution in Corollary~3.4 of Chen
\cite{Ch75}, we let
the $n\times n$ array $\mathbb{Y}$ be such that the first $k$ rows are
independent copies of $\{Y_1,\dots, Y_n\}$ and the other rows are
zeros. Then $\mathcal{L}(\sum_{i=1}^n Y_{i \pi(i)})=\mathcal
{L}(\sum_{i=1}^k Y_{\xi_i})$. Therefore, the bound \eq{c3-1} follows
from Theorem~\ref{t1} and Remark~\ref{c2}.
\end{pf}

Our interest in Corollary~\ref{c3} is motivated by the work of
Wolff \cite{Wo12} who considered sampling without
replacement of independent
random variables in the construction of random embeddings which are
with high probability almost isometric in the sense of Johnson and Lindenstrauss \cite{JoLi84}.
In Corollary~\ref{c3}, if $\mu_i=0$ for all $i\in[n]$, then $\sigma
^2=\frac{k}{n}\sum_{i=1}^n \sigma_i^2$ and the error bound in \eq
{c3-1} reduces to $\frac{451k}{n\sigma^3}\sum_{i=1}^n \E|Y_i|^3$.
Wolff \cite{Wo12} studied this special case and obtained
the bound $\frac
{3k}{n\sigma^3}\sum_{i=1}^n \E|Y_i|^3$ for the Wasserstein distance.
The case where $\sigma_i=0$ for all $i\in[n]$ in Corollary~\ref{c3}
was considered by Goldstein \cite{Go07} using zero-bias
coupling, where a
similar bound with a smaller constant was obtained for the Wasserstein
distance (see Theorem~5.1 of Goldstein \cite{Go07}).

Although we will not consider Wasserstein distance in this paper, we
wish to mention that a Wasserstein distance bound can be obtained for
the normal approximation for $\mathcal{L}(W)$ where $W$ is defined in
Theorem~\ref{t1}. The proof for the bound will not require a
concentration inequality and the bound will have a smaller constant.
Indeed, a bound with a smaller constant has been obtained for the
Wasserstein distance in Theorem~6.1 of Goldstein \cite{Go07}
in the case where
$\sigma_{ij}=0$.

In the next section, we prove a concentration inequality using
exchangeable pairs (Lemma~\ref{l1}) and apply it to random variables
with a combinatorial structure similar to that of $W$ in Theorem~\ref
{t1}. In Section~\ref{sec3}, we prove our main result, Theorem~\ref{t1}, by
Stein's method of exchangeable pairs and the concentration inequality approach.

\section{A concentration inequality for exchangeable pairs}

Assuming the existence of an exchangeable pair $(S, S')$ satisfying an
approximate linearity condition, the next lemma provides a bound on $\P
(S\in[a,b])$.

\begin{lem}\label{l1}
Suppose $(S,S')$ is an exchangeable pair of square integrable random
variables and satisfies the following approximate linearity condition
%
\begin{equation}
\label{l1-1} \E\bigl(S'-S|S\bigr)=-\lambda S+R
\end{equation}
for a positive number $\lambda$ and a random variable $R$. Then, for $a<b$,
%
\begin{eqnarray}
\label{l1-2} &&P\bigl(S\in[a,b]\bigr)\nonumber
\\
&&\quad \le\frac{\E|S|+\E|R|/\lambda}{\E S^2-\E|SR|/\lambda
-1/2}\biggl(\frac
{b-a}{2}+\delta\biggr)
\\
&&\qquad{} +\frac{1}{\E S^2-\E|SR|/\lambda-1/2}\sqrt{\Var \biggl( \E \biggl( \frac{1}{2\lambda}
\bigl(S'-S\bigr)^2I\bigl(\bigl|S'-S\bigr|\le\delta
\bigr) | S \biggr) \biggr)},\nonumber %
\end{eqnarray}
where
%
\begin{equation}
\label{l1-3} \delta=\frac{\E|S'-S|^3}{\lambda}
\end{equation}
provided that $\E S^2-\E|SR|/\lambda-1/2>0$.
\end{lem}

\begin{rem}
Bounding the last term on the right-hand side of \eq{l1-2} involves
studying the conditional distribution of $(S'-S)^2$ given $S$.
Truncating $|S'-S|$ at $\delta$ allows us to keep within third
moments. Shao and Su \cite{ShSu06} has also used a
concentration inequality for
exchangeable pairs in their study of character ratios of the symmetric
group. Their concentration inequality, which is different from ours, is
easier to prove than ours but may not be easy to apply to the
combinatorial central limit theorem.
\end{rem}
\begin{pf*}{Proof of Lemma~\ref{l1}}
Without loss of generality, assume $\delta< \infty$. From the
exchangeability of $S$ and $S'$,
%
\begin{equation}
\label{steinexch} \E\bigl(S'-S\bigr) \bigl(f\bigl(S'
\bigr)+f(S)\bigr)=0
\end{equation}
for all $f$ such that the above expectation exists. Therefore,
\begin{eqnarray*}
\E\bigl(S'-S\bigr) \bigl(f\bigl(S'\bigr)-f(S)\bigr)=2
\E\bigl(S-S'\bigr)f(S).
\end{eqnarray*}
Using the approximate linearity condition \eq{l1-1} for the right-hand
side of the above equation, we have for absolutely continuous $f$,
%
\begin{eqnarray}
\label{l1.0} \E Sf(S)&=&\frac{1}{2\lambda} \E\bigl(S'-S
\bigr) \bigl(f\bigl(S'\bigr)-f(S)\bigr)+\frac{1}{\lambda
} \E Rf(S)\nonumber
\\[-8pt]\\[-8pt]
&=&\frac{1}{2\lambda}\E\bigl(S'-S\bigr)\int_0^{S'-S}
f'(S+t) \,\mathrm{d}t +\frac
{1}{\lambda} \E Rf(S).\nonumber %
\end{eqnarray}
The identity \eq{steinexch} was introduced by Stein \cite
{St86} and \eq
{l1.0} was obtained by Stein \cite{St86} in the case
$R=0$ and by Rinott and Rotar \cite{RiRo97} for $R\ne0$.

Let $f$ be such that $f'(w)=I(a-\delta\le w\le b+\delta)$ and
$f(\frac{a+b}{2})=0$. Therefore, $|f|\le\frac{b-a}{2}+\delta$.
Using the property that for all $w,w'\in\mathbb{R}$,
\begin{eqnarray*}
\bigl(w'-w\bigr)\int_0^{w'-w}
f'(w+t) \,\mathrm{d}t\ge0,\vspace*{-1pt}
\end{eqnarray*}
we have\vspace*{-1pt}
\begin{eqnarray*}
&&\frac{1}{2\lambda}\E\bigl(S'-S\bigr)\int
_0^{S'-S} f'(S+t)\, \mathrm{d}t
\\
&&\quad\ge\frac{1}{2\lambda}\E\bigl(S'-S\bigr)\int
_0^{S'-S} f'(S+t) \,\mathrm{d}t I
\bigl(\bigl|S'-S\bigr|\le \delta\bigr)I\bigl(S\in[a,b]\bigr)
\\
&&\quad= \E I\bigl(S\in[a,b]\bigr)\frac{1}{2\lambda}\bigl(S'-S
\bigr)^2 I\bigl(\bigl|S'-S\bigr|\le \delta\bigr)
\\
&&\quad= \E I\bigl(S\in[a,b]\bigr) \biggl[ \E \biggl( \frac{1}{2\lambda}
\bigl(S'-S\bigr)^2 I\bigl(\bigl|S'-S\bigr|\le\delta
\bigr) | S \biggr)-\E\frac{1}{2\lambda}\bigl(S'-S\bigr)^2
I\bigl(\bigl|S'-S\bigr|\le\delta\bigr) \biggr]
\\
&&\qquad{} +\E I\bigl(S\in[a,b]\bigr)\E\frac{1}{2\lambda}\bigl(S'-S
\bigr)^2 I\bigl(\bigl|S'-S\bigr|\le\delta\bigr)
\\
&&\quad:=-R_1+R_2. %
\end{eqnarray*}
Using the Cauchy--Schwarz inequality,
\begin{eqnarray*}
R_1\le\sqrt{\Var \biggl( \E \biggl( \frac{1}{2\lambda
}
\bigl(S'-S\bigr)^2I\bigl(\bigl|S'-S\bigr|\le\delta
\bigr) | S \biggr) \biggr)}.
\end{eqnarray*}
From \eq{l1-1},
\begin{eqnarray*}
\E\bigl(S'-S\bigr)^2=2\E S(\lambda S-R)=2\lambda\E
S^2 -2\E SR.
\end{eqnarray*}
Therefore,
\begin{eqnarray*}
R_2&=&P\bigl(S\in[a,b]\bigr) \E\frac{1}{2\lambda}
\bigl(S'-S\bigr)^2 -P\bigl(S\in[a,b]\bigr) \E
\frac{1}{2\lambda}\bigl(S'-S\bigr)^2I
\bigl(\bigl|S'-S\bigr|> \delta\bigr)
\\
&\ge& P\bigl(S\in[a,b]\bigr) \biggl(\E S^2 -\frac{\E SR}{\lambda}
\biggr)-P\bigl(S\in [a,b]\bigr)\frac
{1}{\delta} \frac{\E|S'-S|^3}{2\lambda}
\\
&=&P\bigl(S\in[a,b]\bigr) \biggl(\E S^2 -\frac{\E SR}{\lambda}-
\frac{1}{2}\biggr)
\\
&\ge& P\bigl(S\in[a,b]\bigr) \biggl(\E S^2 -\frac{\E|SR|}{\lambda}-
\frac{1}{2}\biggr), %
\end{eqnarray*}
where in the last equality we used the definition of $\delta$ in \eq
{l1-3}. Using the fact that $|f|\le\frac{b-a}{2}+\delta$, we have
\begin{eqnarray*}
\bigl|\E Sf(S)\bigr|\le\biggl(\frac{b-a}{2}+\delta\biggr)\E|S|,\qquad  \biggl|\frac{1}{\lambda}\E
Rf(S)\biggr|\le\biggl(\frac{b-a}{2}+\delta\biggr)\frac{\E|R|}{\lambda}.
\end{eqnarray*}
The lemma is proved by applying all the above bounds to \eq{l1.0}.
\end{pf*}
Now we apply Lemma~\ref{l1} to establish a concentration inequality
for a sum $S$ which is defined as follows. Let $\mathbb{X}$ be the $n
\times n$ array defined in Theorem~\ref{t1} satisfying \eq{1.1} and
%
\begin{equation}
\label{1.2a} \frac{1}{n} \sum_{i,j=1}^n
\sigma_{ij}^2 +\frac{1}{n-1} \sum
_{i,j=1}^n c_{ij}^2=1.
\end{equation}
For $n\ge6$ and $m\in\{2,3,4\}$, we remove the last $m$ rows and
columns from the original array $\mathbb{X}$. Let $\tau$ be an
independent uniform random permutation of $[n-m]$.
Define the variable $S$ by
%
\begin{equation}
\label{2.14} S=\sum_{i=1}^{n-m}
X_{i\tau(i)}.
\end{equation}
To prove a concentration inequality for $S$, we need the following
lemma which estimates the second moment of $S$.

\begin{lem}\label{momentlemma}
Let $S$ be defined by \eq{2.14} for some $m\in\{2,3,4\}$ and $n\ge
6$. Suppose $\gamma\le1/c_0$ where $\gamma$ was defined in \eq
{t1-2}. Under the assumptions \eq{1.1} and \eq{1.2a}, we have
%
\begin{equation}
\label{ES2} \frac{n-1}{n-2} -\frac{2n}{(n-4)c_0^{2/3}} -\frac{24n}{(n-5)^2}\le \E
S^2\le\frac{n}{n-5}+\frac{24n}{(n-5)^2}.
\end{equation}
\end{lem}

\begin{pf}
Writing $\sum_{1\le i\ne j\le n-m}\sum_{1\le k\ne l\le n-m}$ as
\begin{eqnarray*}
\sum_{1\le i,j,k,l \le n-m}-\sum_{1\le i=j\le n-m}
\sum_{1\le k,l\le
n-m}-\sum_{1\le i,j\le n-m}\sum
_{1\le k=l\le n-m}+\sum_{1\le i=j\le
n-m}\sum
_{1\le k=l\le n-m}
\end{eqnarray*}
and using the assumption \eq{1.1},
%
\begin{eqnarray}
\label{210} \E S^2 &=& \E\Biggl(\sum_{i=1}^{n-m}X_{i\tau(i)}
\Biggr)^2=\sum_{i=1}^{n-m} \E
X_{i\tau
(i)}^2 +\sum_{1\le i\ne j\le n-m} \E
X_{i\tau(i)}X_{j\tau(j)}\nonumber
\\
&=&\frac{1}{n-m}\sum_{i,j=1}^{n-m}
\E X_{ij}^2 +\frac
{1}{(n-m)_{(2)}}\sum
_{1\le i\ne j \le n-m} \sum_{1\le k\ne l\le n-m} \E
X_{ik}X_{jl}\nonumber
\\
&=&\frac{1}{n-m}\sum_{i,j=1}^{n-m}
\bigl(c_{ij}^2+\sigma_{ij}^2\bigr)+
\frac
{1}{(n-m)_{(2)}}\sum_{1\le i\ne j\le n-m} \sum
_{1\le k\ne l\le n-m} c_{ik}c_{jl}
\nonumber\\
&=&\frac{1}{n-m}\sum_{i,j=1}^{n-m}
\sigma_{ij}^2+\frac{1}{n-m}\sum
_{i,j=1}^{n-m}c_{ij}^2
\\
&&{} +\frac{1}{(n-m)_{(2)}}\sum_{i,k=1}^{n-m}
c_{ik} \Biggl(\sum_{j,l=1}^{n-m}c_{jl}-
\sum_{l=1}^{n-m}c_{il}-\sum
_{j=1}^{n-m} c_{jk}+c_{ik}\Biggr)
\nonumber\\
&=&\frac{1}{n-m}\sum_{i,j=1}^{n-m}
\sigma_{ij}^2+\frac{1}{n-m-1}\sum
_{i,j=1}^{n-m}c_{ij}^2\nonumber
\\
&&{} +\frac{1}{(n-m)_{(2)}}\sum_{i,j=1}^{n-m}
c_{ij}\Biggl(\sum_{k,l=n-m+1}^n
c_{kl}+\sum_{k=n-m+1}^n
c_{ik}+\sum_{l=n-m+1}^n
c_{lj} \Biggr)\nonumber
\end{eqnarray}
with the falling factorial notation $(n-m)_{(2)}:=(n-m)(n-m-1)$. Under
the assumption \eq{1.2a}, $\E S^2$ is close to $1$ intuitively. We
quantify it as follows.
From \eq{1.1} and \eq{1.2a},
\begin{eqnarray*}
\Biggl|\sum_{i,j=1}^{n-m}c_{ij}\Biggl(\sum
_{k=n-m+1}^n c_{ik}\Biggr)\Biggr|=\Biggl|\sum
_{i=1}^{n-m}\Biggl(\sum
_{k=n-m+1}^n c_{ik}\Biggr)^2\Biggr|\le\Biggl|
\sum_{i=1}^{n-m} m\sum
_{k=n-m+1}^n c_{ik}^2\Biggr|\le m(n-1).
\end{eqnarray*}
Similarly,
\begin{eqnarray*}
\Biggl|\sum_{i,j=1}^{n-m} c_{ij} \sum
_{l=n-m+1}^n c_{lj}\Biggr|\le m(n-1).
\end{eqnarray*}
Moreover,
\begin{eqnarray*}
\Biggl|\sum_{i,j=1}^{n-m}c_{ij}\sum
_{k,l=n-m+1}^n c_{kl} \Biggr|=\Biggl|\Biggl(\sum
_{k,l=n-m+1}^n c_{kl}
\Biggr)^2\Biggr|\le m^2\sum_{k,l=n-m+1}^n
c_{kl}^2 \le m^2 (n-1).
\end{eqnarray*}
Therefore, from \eq{210},
\begin{eqnarray*}
\E S^2\le\frac{1}{n-m-1} \sum_{i,j=1}^n
\bigl(\sigma_{ij}^2+c_{ij}^2\bigr) +
\frac{(2m+m^2)(n-1)}{(n-m)_{(2)}}\le\frac{n}{n-5}+\frac{24n}{(n-5)^2}.
\end{eqnarray*}
Now we prove the lower bound. Since $\gamma\le1/c_0$, using H\"
older's inequality, we have
\begin{eqnarray*}
\sum_{1\le i, j\le n: i> n-m \atop\mathrm{or\ } j> n-m} \bigl(\sigma _{ij}^2+c_{ij}^2
\bigr)=\sum_{1\le i, j\le n: i> n-m \atop\mathrm{or\ }  j>
n-m} \E X_{ij}^2
\le(8n)^{1/3} \Biggl(\sum_{i,j=1}^n
\E|X_{ij}|^3\Biggr)^{2/3}\le 2n/c_0^{2/3}.
\end{eqnarray*}
Therefore,
\begin{eqnarray*}
\E S^2 &\ge&\frac{1}{n-m} \sum
_{i,j=1}^{n-m} \bigl(\sigma _{ij}^2+c_{ij}^2
\bigr)-\frac{24n}{(n-5)^2}
\\
&=&\frac{1}{n-m} \sum_{i,j=1}^n \bigl(
\sigma_{ij}^2+c_{ij}^2\bigr)-
\frac
{1}{n-m} \sum_{1\le i, j\le n: i> n-m \atop\mathrm{or\ }  j> n-m} \bigl(
\sigma_{ij}^2+c_{ij}^2\bigr) -
\frac{24n}{(n-5)^2}
\\
&\ge&\frac{n-1}{n-2} -\frac{2n}{(n-4)c_0^{2/3}} -\frac{24n}{(n-5)^2}. %
\end{eqnarray*}
\upqed
\end{pf}
In the next proposition, we provide a concentration inequality for $S$.

\begin{prop}\label{5-p1}
Let $S$ be defined by \eq{2.14} for some $m\in\{2,3,4\}$. Suppose
$\gamma\le1/c_0$ where $\gamma$ was defined in \eq{t1-2}, and $c_0$
and $n\ge6$ are large enough to satisfy
%
\begin{equation}
\label{5-p1-0} \theta:=\frac{1}{2}-\frac{2n}{(n-4)c_0^{2/3}}-
\frac
{24n}{(n-5)^2}-\frac{4\sqrt{n}}{n-4}\sqrt{\frac{n}{n-5}+
\frac
{24n}{(n-5)^2}}>0. %
\end{equation}
Then for all $a<b$,
%
\begin{equation}
\label{5-p1-1} \P\bigl(S\in[a,b]\bigr) \le c_1 (b-a) +c_2
\gamma,
\end{equation}
where
%
\begin{equation}
\label{5-c1} c_1= \biggl( \frac{1}{2}\sqrt{
\frac{n}{n-5}+\frac
{24n}{(n-5)^2}}+\frac{2\sqrt{n}}{n-4} \biggr) \bigl/\theta
\end{equation}
and
%
\begin{equation}
\label{5-c2} c_2=\frac{64n}{n-4}c_1+ \biggl\{
\biggl[ \frac{8n}{(n-4)^2}+\frac
{16n}{n-4}+32\biggl(\frac{n}{n-4}
\biggr)^3 \biggr] \biggl[ \frac{32n}{n-4} \biggr] \biggr
\}^{1/2} \bigl/ \theta.
\end{equation}
\end{prop}

\begin{pf}
For any $m\in\{2,3,4\}$, we construct an exchangeable pair $(S,S')$ by
uniformly selecting two different indices $I,J\in[n-m]$ and letting
$S'=S-X_{I\tau(I)}-X_{J\tau(J)}+X_{I\tau(J)}+X_{J\tau(I)}$. An
approximate linearity condition with an error term can be established as
%
\begin{eqnarray}
\label{1.6} &&\E\bigl(S'-S|S\bigr)\nonumber
\\
&&\quad=\frac{1}{(n-m)_{(2)}}\sum_{1\le i, j \le n-m}\E \bigl\{
[X_{i\tau(j)}+X_{j\tau(i)} ]- [X_{i\tau(i)}+X_{j\tau
(j)} ] |S
\bigr\}\nonumber
\\[-8pt]\\[-8pt]
&&\quad=\frac{1}{(n-m)_{(2)}} \E \Biggl\{ 2\sum_{i,j=1}^{n-m}X_{ij}
-2(n-m)S \Bigl| S \Biggr\}\nonumber
\\
&&\quad=-\lambda S+R, \nonumber%
\end{eqnarray}
where $\lambda=2/(n-m-1)$ and
\begin{eqnarray*}
R=\frac{2}{(n-m)_{(2)}}\E\Biggl(\sum_{i,j=1}^{n-m}
X_{ij}\Bigl|S\Biggr).
\end{eqnarray*}
To apply the concentration inequality in Lemma~\ref{l1}, we need to:
\begin{enumerate}[3.]
\item[1.] Bound $\sqrt{\E R^2}/\lambda$.

\item[2.] Bound
%
\begin{equation}
\label{209} \delta=\frac{\E|S'-S|^3}{\lambda}.
\end{equation}

\item[3.] Bound\vspace*{-1pt}
\begin{eqnarray*}
B_0=\sqrt{\Var \biggl( \E \biggl( \frac{1}{2\lambda}
\bigl(S'-S\bigr)^2 I\bigl(\bigl|S'-S\bigr|\le\delta
\bigr) \bigl| S \biggr) \biggr)}.\vspace*{-1pt}
\end{eqnarray*}
\end{enumerate}

First,\vspace*{-2pt}
\begin{eqnarray*}
&&\sqrt{\E R^2}
\\[-1pt]
&&\quad= \frac{2}{(n-m)_{(2)}}\sqrt{\E\Biggl(\E\Biggl(\sum
_{i,j=1}^{n-m} X_{ij}\Bigl|S\Biggr)
\Biggr)^2}\\[-1pt]
&&\quad  = \frac{2}{(n-m)_{(2)}} \sqrt{\Var\Biggl(\E\Biggl(
\sum_{i,j=1}^{n-m}X_{ij}\Bigl|S\Biggr)
\Biggr)+\Biggl(\E\sum_{i,j=1}^{n-m}X_{ij}
\Biggr)^2}
\\[-1pt]
&&\quad\le\frac{2}{(n-m)_{(2)}} \sqrt{\Var\Biggl(\sum
_{i,j=1}^{n-m}X_{ij}\Biggr)+\Biggl(\sum
_{i,j=1}^{n-m}c_{ij}\Biggr)^2}\\[-1pt]
&&\quad  =
\frac
{2}{(n-m)_{(2)}} \sqrt{\sum_{i,j=1}^{n-m}
\sigma_{ij}^2+\Biggl(\sum_{i,j=n-m+1}^{n}c_{ij}
\Biggr)^2}
\\[-1pt]
&&\quad\le\frac{2}{(n-m)_{(2)}} \sqrt{\sum_{i,j=1}^{n-m}
\sigma _{ij}^2+m^2\sum
_{i,j=n-m+1}^{n}c_{ij}^2} \\[-1pt]
&&\quad \le
\frac
{2}{(n-m)_{(2)}}\sqrt{m^2\Biggl(\sum
_{i,j=1}^n \sigma_{ij}^2 +\sum
_{i,j=1}^n c_{ij}^2
\Biggr)}
\le\frac{2m\sqrt{n}}{(n-m)_{(2)}}, %
\end{eqnarray*}
where we used the assumptions \eq{1.1} and \eq{1.2a}. Therefore,\vspace*{-1pt}
%
\begin{equation}
\label{2110} \frac{\sqrt{\E R^2}}{\lambda}\le\frac{4\sqrt{n}}{n-4}.
\end{equation}
Next, we bound $\delta$ of \eq{209}.
From the fact that\vspace*{-1pt}
%
\begin{eqnarray}
\label{2.31} && |X_{i\tau(j)}+X_{j\tau(i)}-X_{i\tau(i)}-X_{j\tau(j)}|^3\nonumber
\\[-9pt]\\[-9pt]
&&\quad\le16 \bigl( |X_{i\tau(j)}|^3+ |X_{j\tau(i)}|^3+
|X_{i\tau(i)}|^3+ |X_{j\tau(j)}|^3\bigr),\nonumber
\end{eqnarray}
we have
\begin{eqnarray*}
&&\E|S'-S|^3
\\
&&\quad=\E\frac{1}{(n-m)_{(2)}}\sum_{1\le i, j \le n-m}\E \bigl(
|X_{i\tau(i)}+X_{j\tau(j)}-X_{i\tau(j)}-X_{j\tau(i)}|^3
| \mathbb{X} \bigr)
\\
&&\quad\le\E\frac{16}{(n-m)_{(2)}} \sum_{1\le i, j\le n-m}\E \bigl( {
\bigl(}|X_{i\tau(i)}|^3+|X_{j\tau(j)}|^3+|X_{i\tau(j)}|^3+|X_{j\tau
(i)}|^3
{\bigr)} | \mathbb{X} \bigr)
\\
&&\quad\le\frac{64n}{(n-m)_{(2)}}\gamma. %
\end{eqnarray*}
Therefore,
%
\begin{equation}
\label{1.7} \delta\le\frac{32n}{n-4} \gamma.
\end{equation}
Now we turn to the final step of bounding $B_0$. Denote
\begin{eqnarray*}
\alpha_{ij}^{\tau}=(X_{i\tau(i)}+X_{j \tau(j)}-X_{i\tau
(j)}-X_{j\tau(i)})^2I\bigl(|X_{i\tau(i)}+X_{j \tau(j)}-X_{i\tau
(j)}-X_{j\tau(i)}|
\le\delta\bigr).
\nonumber
\end{eqnarray*}
We have
\begin{eqnarray*}
\frac{1}{2\lambda} \E \bigl( \bigl(S'-S\bigr)^2I
\bigl(\bigl|S'-S\bigr|\le\delta\bigr) | \mathbb{X},\tau \bigr)=
\frac{1}{4(n-m)} \sum_{1\le i\ne j \le n-m} \alpha_{ij}^{\tau}.
\end{eqnarray*}
Therefore, with $|\cdot|$ denoting cardinality when applied to a
subset of $[n]$,
\begin{eqnarray*}
B_0^2&=&\Var \biggl( \E \biggl(
\frac{1}{2\lambda}\bigl(S'-S\bigr)^2 I
\bigl(\bigl|S'-S\bigr|\le\delta\bigr) | S \biggr) \biggr)
\\
&\le&\Var \biggl( \frac{1}{2\lambda} \E \bigl( \bigl(S'-S
\bigr)^2I\bigl(\bigl|S'-S\bigr|\le \delta\bigr) | \mathbb{X},\tau
\bigr)\biggr)
\\
&=&\frac{1}{16(n-m)^2} \biggl\{ 2\sum_{1\le i\ne j \le n-m} \Var
\bigl(\alpha_{ij}^\tau\bigr) +\sum
_{1\le i,j,i',j'\le n-m,\atop i\ne j, i'\ne
j', |\{i,j,i',j'\}|=3} \cov\bigl(\alpha_{ij}^{\tau},
\alpha_{i'j'}^{\tau
}\bigr)
\\
&&\hphantom{\frac{1}{16(n-m)^2} \biggl\{}{} + \sum_{1\le i,j,i',j'\le n-m,\atop|\{i,j,i',j'\}|=4} \cov \bigl(\alpha
_{ij}^{\tau}, \alpha_{i'j'}^{\tau}\bigr) {
\biggr\}}
\\
&:=& R_1 +R_2+R_3. %
\end{eqnarray*}
The terms $R_1$ and $R_2$ are easy to bound.
%
\begin{equation}
\label{216} |R_1|\le\frac{2}{16(n-m)^2}\sum
_{i,j=1}^{n-m} \delta\E|X_{i\tau
(i)}+X_{j\tau(j)}-X_{i\tau(j)}-X_{j\tau(i)}|^3
\le\frac{8n\delta
}{(n-4)^2}\gamma.
\end{equation}
From $\cov(X, Y)\le(\Var(X)+\Var(Y))/2$, \eq{2.31} and the
restriction that $i\ne j, i'\ne j',\allowbreak  |\{i,j,i',j'\}|=3$, we have
%
\begin{eqnarray}
\label{217} |R_2| &\le&\frac{\delta}{16(n-m)^2} \sum
_{1\le i,j,i',j'\le
n-m,\atop i\ne j, i'\ne j', |\{i,j,i',j'\}|=3} \E|X_{i\tau
(i)}+X_{j\tau(j)}-X_{i\tau(j)}-X_{j\tau(i)}|^3\nonumber
\\
&\le&\frac{\delta}{(n-m)^2} \sum_{1\le i,j,i',j'\le n-m,\atop i\ne
j, i'\ne j', |\{i,j,i',j'\}|=3} \bigl(
\E|X_{i\tau(i)}|^3+\E|X_{j\tau
(j)}|^3+
\E|X_{i\tau(j)}|^3 + \E|X_{j\tau(i)}|^3 \bigr)\nonumber
\\[-8pt]\\[-8pt]
&=& \frac{8\delta(n-m-2)}{(n-m)^2} \Biggl[ (n-m-1) \sum_{i=1}^{n-m}
\E |X_{i\tau(i)}|^3 + \sum_{1\le i\ne j\le n-m}
\E|X_{i\tau(j)}|^3 \Biggr]\nonumber
\\
&\le&\frac{16n\delta}{n-4}\gamma. \nonumber%
\end{eqnarray}
Let $\alpha
_{ij}^{kl}=(X_{ik}+X_{jl}-X_{il}-X_{jk})^2I(|X_{ik}+X_{jl}-X_{il}-X_{jk}|\le
\delta)$. For $|\{i,j,i',j'\}|=4$,
\begin{eqnarray*}
\cov\bigl(\alpha_{ij}^\tau,
\alpha_{i'j'}^\tau\bigr)&=& \E\alpha_{ij}^\tau
\alpha_{i'j'}^\tau-\E\alpha_{ij}^\tau\E
\alpha_{i'j'}^\tau\\
&=& \frac{1}{(n-m)_{(4)}}\sum
_{1\le k,l,k',l'\le n-m,\atop|\{k,l,k',l'\}
|=4} \E\alpha_{ij}^{kl}
\alpha_{i'j'}^{k'l'}
\\
&&{}- \biggl[ \frac{1}{(n-m)_{(2)}} \sum_{1\le k\ne l\le n-m} \E
\alpha_{ij}^{kl} \biggr] \biggl[\frac{1}{(n-m)_{(2)}}\sum
_{1\le k'\ne
l'\le n-m}\E\alpha_{i'j'}^{k'l'} \biggr]
\\
&=& \frac{1}{(n-m)_{(4)}}\sum_{1\le k,l,k',l'\le n-m,\atop|\{
k,l,k',l'\}|=4} \E
\alpha_{ij}^{kl}\E\alpha_{i'j'}^{k'l'} \\
&&{}-
\frac
{1}{[(n-m)_{(2)}]^2}\sum_{1\le k,l,k',l'\le n-m,\atop|\{k,l,k',l'\}
|=4} \E\alpha_{ij}^{kl}
\E\alpha_{i'j'}^{k'l'}
\\
&&{} -\frac{1}{[(n-m)_{(2)}]^2}\sum_{1\le k,l,k',l'\le n-m,\atop k\ne l,
k'\ne l', |\{k,l,k',l'\}|\le3} \E
\alpha_{ij}^{kl}\E\alpha _{i'j'}^{k'l'}
\\
&=& \frac{4(n-m)-6}{(n-m)_{(2)}(n-m)_{(4)}} \sum_{1\le k,l,k',l'\le
n-m,\atop|\{k,l,k',l'\}|=4} \E
\alpha_{ij}^{kl}\E\alpha _{i'j'}^{k'l'}
\\
&&{} -\frac{1}{[(n-m)_{(2)}]^2}\sum_{1\le k,l,k',l'\le n-m,\atop k\ne l,
k'\ne l', |\{k,l,k',l'\}|\le3} \E
\alpha_{ij}^{kl}\E\alpha_{i'j'}^{k'l'}.
\end{eqnarray*}
Therefore,
%
\begin{eqnarray}
\label{218} |R_3|&\le&\frac{1}{16(n-m)^2} \sum
_{1\le i,j,i',j'\le n-m,\atop|\{
i,j,i',j'\}|=4} \biggl[ \frac{4(n-m)-6}{(n-m)_{(2)}(n-m)_{(4)}} \sum
_{1\le k,l,k',l'\le n-m,\atop|\{k,l,k',l'\}|=4} \frac{\E(\alpha
_{ij}^{kl})^2+\E(\alpha_{i'j'}^{k'l'})^2}{2}\nonumber
\\
& &\hphantom{\frac{1}{16(n-m)^2} \sum
_{1\le i,j,i',j'\le n-m,\atop|\{
i,j,i',j'\}|=4} \biggl[}{}+\frac{1}{[(n-m)_{(2)}]^2}\sum_{1\le k,l,k',l'\le n-m,\atop k\ne l,
k'\ne l', |\{k,l,k',l'\}|\le3}
\frac{\E(\alpha_{ij}^{kl})^2+\E
(\alpha_{i'j'}^{k'l'})^2}{2} \biggr]\nonumber
\\
&\le&\frac{1}{16(n-m)^2} \nonumber \\
&&{}\times\sum_{1\le i,j,i',j'\le n-m,\atop|\{
i,j,i',j'\}|=4}\biggl[
\frac{4}{(n-m)(n-m-1)^2}\nonumber
\\
&&\hphantom{{}\times\sum_{1\le i,j,i',j'\le n-m,\atop|\{
i,j,i',j'\}|=4}\biggl[}{} \times\biggl(\sum_{1\le k\ne l\le n-m} \E\bigl(
\alpha_{ij}^{kl}\bigr)^2/2+\sum
_{1\le k'\ne l'\le n-m} \E\bigl(\alpha_{i'j'}^{k'l'}\bigr)/2
\biggr)
\\
&&\hphantom{{}\times\sum_{1\le i,j,i',j'\le n-m,\atop|\{
i,j,i',j'\}|=4}\biggl[}{} + \frac{4}{(n-m)(n-m-1)^2}\nonumber\\
&&\hphantom{{}\times\sum_{1\le i,j,i',j'\le n-m,\atop|\{
i,j,i',j'\}|=4}\biggl[{}+}{}\times\biggl(\sum_{1\le k\ne l\le n-m} \E\bigl(
\alpha _{ij}^{kl}\bigr)^2/2 +\sum
_{1\le k'\ne l'\le n-m} \E\bigl(\alpha _{i'j'}^{k'l'}
\bigr)^2/2\biggr) \biggr]\nonumber
\\[-1pt]
&\le&\frac{1}{2(n-m)^3(n-m-1)^2} \sum_{1\le i,j,i',j'\le n-m,\atop
|\{
i,j,i',j'\}|=4} \sum
_{1\le k\ne l\le n-m} \E\bigl(\alpha_{ij}^{kl}
\bigr)^2\nonumber
\\[-1pt]
&\le&\frac{\delta}{2(n-m)^3} \sum_{i,j,k,l=1}^n \E
|X_{ik}+X_{jl}-X_{il}-X_{jk}|^3
\le32\biggl(\frac{n}{n-4}\biggr)^3\delta\gamma.\nonumber %
\end{eqnarray}
From \eq{216}, \eq{217}, \eq{218}, and then applying \eq{1.7}, we obtain\vspace*{-1pt}
%
\begin{eqnarray}
\label{2112} B_0^2&\le& \biggl[
\frac{8n}{(n-4)^2}+\frac{16n}{n-4}+32\biggl(\frac
{n}{n-4}
\biggr)^3 \biggr] \delta\gamma\nonumber
\\[-8pt]\\[-8pt]
&\le& \biggl[ \frac{8n}{(n-4)^2}+\frac{16n}{n-4}+32\biggl(\frac{n}{n-4}
\biggr)^3 \biggr] \frac{32n \gamma^2}{n-4}.\nonumber %
\end{eqnarray}
Now we are ready to obtain a concentration inequality for $S$ using
Lemma~\ref{l1}. From \eq{l1-2}, and applying the bounds \eq{2110},
\eq{ES2}, \eq{1.7}, \eq{2112}, we obtain\vspace*{-1pt}
\begin{eqnarray*}
&&P\bigl(S\in[a,b]\bigr)\\
&& \quad\le\Biggl(\biggl(\sqrt{\frac{n}{n-5}+\frac
{24n}{(n-5)^2}}+\frac{4\sqrt{n}}{n-4}\biggr)\\
&&\qquad\hphantom{\Biggl({}}{}\Bigl/\biggl(\frac{n-1}{n-2}-\frac
{2n}{(n-4)c_0^{2/3}}-\frac{24n}{(n-5)^2}\\
&&\qquad\hphantom{\Biggl({}\Bigl/\biggl(}{}-\frac{4\sqrt{n}}{n-4}\sqrt {\frac{n}{n-5}+\frac{24n}{(n-5)^2}}-\frac{1}{2}\biggr)\Biggr)\biggl(
\frac
{b-a}{2}+\delta\biggr)
\\
& &\qquad{}+\Biggl(\biggl(\sqrt{\frac{8n}{(n-4)^2}+\frac{16n}{n-4}+32\biggl(\frac
{n}{n-4}\biggr)^3}\sqrt{\frac{32n}{n-4}}\biggr)\\
&&\qquad\hphantom{{}+\Biggl(}{}\Bigl/\biggl(\frac{n-1}{n-2}-\frac
{2n}{(n-4)c_0^{2/3}}-\frac{24n}{(n-5)^2}-\frac{4\sqrt{n}}{n-4}\sqrt {\frac{n}{n-5}+\frac{24n}{(n-5)^2}}-\frac{1}{2}\biggr)\Biggr) \gamma
\\
&&\quad\le c_1(b-a) + c_2 \gamma. %
\end{eqnarray*}
\upqed
\end{pf}

\begin{rem}
Lemma~\ref{momentlemma} and Proposition~\ref{5-p1} still hold if S is
defined similarly as in (\ref{2.14}) but with any $m$ rows and $m$ columns removed
from the array $X$ where $m = {2, 3, 4}$.
\end{rem}

\begin{rem}
From Proposition~\ref{5-p1}, the error in Neammanee and Suntornchost \cite
{NeSu05} can be
corrected by conditioning on (using their notation)\vspace*{-1pt}
\begin{eqnarray*}
&J,K,L,M,\tau(J),\tau(K),\tau(L),\tau(M),&
\\
&\bigl\{\hat{X}_{ij}\dvt  i\in\bigl\{J,K,\tau^{-1}(L),
\tau^{-1}(M)\bigr\}, j\in\bigl\{ L,M,\tau(J),\tau(K)\bigr\}\bigr\},&
\end{eqnarray*}
and by applying our Proposition~\ref{5-p1} instead of their
Proposition~2.7.
\end{rem}

\section{Proof of the main result}\label{sec3}

From \eq{1.1}, $\E W=0$. The variance of $W$ can be calculated as
follows. From \eq{1.1},
\begin{eqnarray*}
\Var(W)&=&\Var\Biggl(\sum_{i=1}^n
X_{i\pi(i)}\Biggr)
\\
&=&\sum_{i=1}^n
\Var(X_{i\pi(i)})+\sum_{1\le i\ne j\le n}\cov
(X_{i\pi(i)}, X_{j\pi(j)})
\\
&=& \sum_{i=1}^n
\E(X_{i\pi(i)}-c_{i\cdot})^2+\sum
_{1\le
i\ne j\le
n} \E(X_{i\pi(i)}-c_{i\cdot})
(X_{j\pi(j)}-c_{j\cdot})
\\
&=&\frac{1}{n} \sum_{i,j=1}^n
\E(X_{ij}-c_{i\cdot})^2+\frac
{1}{n(n-1)}\sum
_{1\le i\ne j\le n}\sum_{1\le k\ne l\le n} \E
(X_{ik}-c_{i\cdot}) (X_{jl}-c_{j\cdot})
\\
&=&\frac{1}{n}\sum_{i,j=1}^n
\bigl(\sigma_{ij}^2+c_{ij}^2\bigr)+
\frac
{1}{n(n-1)}\sum_{1\le i\ne j\le n}\sum
_{1\le k\ne l \le
n}c_{ik}c_{jl}
\\
&=&\frac{1}{n}\sum_{i,j=1}^n
\bigl(\sigma_{ij}^2+c_{ij}^2\bigr)+
\frac
{1}{n(n-1)}\sum_{i,j=1}^n
c_{ij}^2
\\
&=&\frac{1}{n}\sum_{i,j=1}^n
\sigma_{ij}^2+\frac{1}{n-1}\sum
_{i,j=1}^n c_{ij}^2.
\end{eqnarray*}
This proves the first part of the theorem. In the following, we work
under the assumption that $\Var(W)=1$, that is,
%
\begin{equation}
\label{1.2b} \frac{1}{n}\sum_{i,j=1}^n
\sigma_{ij}^2+\frac{1}{n-1}\sum
_{i,j=1}^n c_{ij}^2=1.
\end{equation}
We assume $\gamma\le1/451$, that is, $c_0=451$ in Proposition~\ref
{5-p1}. Otherwise the bound \eq{t1-1} is obviously true. From \eq
{1.2b} and H\"older's inequality, we have\vspace*{1pt}
%
\begin{equation}
\label{301} n-1\le\sum_{i,j=1}^n \E
X_{ij}^2 \le n^{2/3}\Biggl(\sum
_{i,j=1}^n \E |X_{ij}|^3
\Biggr)^{2/3}=n^{4/3}\gamma^{2/3}.
\end{equation}
Therefore it suffices to prove Theorem~\ref{t1} for $n\geq203\,000$.
For $n\geq203\,000$ and $\gamma\le1/451$, \eq{5-p1-0} is satisfied,
and the concentration inequality \eq{5-p1-1} in Proposition~\ref
{5-p1} is applicable.

We follow the notation in Section~\ref{sec1} and construct an exchangeable pair
$(W,W')$ by uniformly selecting two different indices $I,J\in[n]$ (the
ranges of $I$ and $J$ are different from those in the proof of
Proposition~\ref{5-p1}) and let $W'=W-X_{I\pi(I)}-X_{J\pi
(J)}+X_{I\pi(J)}+X_{J\pi(I)}$. Following the argument as in \eq
{1.6}, we have\vspace*{1pt}
%
\begin{equation}
\label{1.3} \E\bigl(W'-W|W\bigr)=-\lambda W +R,
\end{equation}
where $\lambda=2/(n-1)$ and\vspace*{1pt}
\begin{eqnarray*}
R=\frac{2}{n(n-1)}\E\Biggl(\sum_{i,j=1}^n
X_{ij}\Bigl|W\Biggr).
\end{eqnarray*}
The following bound on $\sqrt{\E R^2}$\vspace*{1pt}
%
\begin{equation}
\label{1.4} \sqrt{\E R^2}\le\frac{2}{n(n-1)} \sqrt{\Var
\Biggl(\sum_{i,j=1}^n X_{ij}
\Biggr)}=\frac{2}{n(n-1)}\sqrt{\sum_{i,j=1}^n
\sigma_{ij}^2}\le \frac{2}{(n-1)\sqrt{n}}
\end{equation}
is obtained by using the assumptions \eq{1.1} and \eq{1.2b}.

From the fact that $(W,W')$ is an exchangeable pair and satisfies an
approximate linearity condition \eq{1.3}, the following functional
identity can be proved by the same argument as in~\eq{l1.0}.
%
\begin{equation}
\label{320} \E Wf(W)=\frac{1}{2\lambda} \E\bigl(W'-W\bigr)
\bigl(f\bigl(W'\bigr)-f(W)\bigr)+\frac{\E
Rf(W)}{\lambda}.
\end{equation}
Let $f$ be the bounded solution to the Stein equation\vspace*{1pt}
%
\begin{equation}
\label{1.9} f'(w)-wf(w)=I(w\le z)-\Phi(z).
\end{equation}
It is known that (Chen and Shao \cite{ChSh05})\vspace*{1.5pt}
%
\begin{equation}
\label{1.17} \bigl|f(w)\bigr|\le\frac{\sqrt{2\uppi}}{4},\qquad  \bigl|f'(w)\bigr|\le1 \qquad \forall w\in
\mathbb{R}
\end{equation}
and
%
\begin{equation}
\label{1.16} \bigl|(w+u)f(w+u)-(w+v)f(w+v)\bigr|\le\biggl(|w|+\frac{\sqrt{2\uppi}}{4}\biggr)
\bigl(|u|+|v|\bigr).
\end{equation}
From \eq{1.9} and \eq{320}, what we need to bound is
\begin{eqnarray*}
&&\P(W\le z)-\Phi(z)
\\
&&\quad=\E f'(W)-\E Wf(W)
\\
&&\quad=\E f'(W) \biggl(1-\frac{(W'-W)^2}{2\lambda}\biggr)+
\frac{1}{2\lambda}\E \bigl(W'-W\bigr)\int_0^{W'-W}
\bigl(f'(W)-f'(W+t)\bigr)\,\mathrm{d}t
\\
&&\qquad{} -\frac{\E Rf(W)}{\lambda}
\\
&&\quad:=R_1+R_2-R_3. %
\end{eqnarray*}
From \eq{1.4} and \eq{1.17}, and recalling $\lambda=2/(n-1)$, we have
%
\begin{equation}
\label{R3} |R_3|\le1/\sqrt{n}.
\end{equation}
To bound $R_1$ and $R_2$, we need the concentration inequality obtained
in the last section.

From \eq{1.3} and \eq{1.4},
%
\begin{equation}
\label{370} \E\bigl(W'-W\bigr)^2=2\lambda-2\E WR
\cases{ \le2\lambda+2\sqrt{\E R^2}\le
\displaystyle \frac{4}{n-1}\biggl(1+\frac{1}{\sqrt {n}}\biggr),
\vspace*{4pt}\cr
\ge2\lambda-2\sqrt{\E R^2}\ge\displaystyle \frac{4}{n-1}\biggl(1-
\displaystyle \frac{1}{\sqrt{n}}\biggr). } %
\end{equation}

We bound $R_2$ first. From \eq{1.9},
\begin{eqnarray*}
R_2&=& \frac{1}{2\lambda} \E\bigl(W'-W
\bigr)\int_0^{W'-W} \bigl(Wf(W)-(W+t)f(W+t)\bigr)\,\mathrm{d}t
\\
&&{} +\frac{1}{4n} \sum_{1\le i\ne j\le n}
\E(X_{i\pi(j)}+X_{j\pi
(i)}-X_{i\pi(i)}-X_{j\pi(j)})
\\
&&\hphantom{{} +\frac{1}{4n} \sum_{1\le i\ne j\le n}}{} \times\int_0^{X_{i\pi(j)}+X_{j\pi(i)}-X_{i\pi
(i)}-X_{j\pi(j)}} \bigl[ I(U\le
z-X_{i\pi(i)}-X_{j\pi(j)})
\\
&&\hphantom{{} +\frac{1}{4n} \sum_{1\le i\ne j\le n}{} \times\int_0^{X_{i\pi(j)}+X_{j\pi(i)}-X_{i\pi
(i)}-X_{j\pi(j)}} \bigl[}{} -I(U\le z-X_{i\pi(i)}-X_{j\pi(j)}-t) \bigr] \,\mathrm{d}t
\\
&:=& R_{2,1} +R_{2,2}, %
\end{eqnarray*}
where $U=\sum_{k\notin\{i,j\}} X_{k\pi(k)}$. Noting that $U$ is
independent of $\{X_{i\pi(i)}, X_{j\pi(j)}, X_{i\pi(j)}, X_{j\pi
(i)}\}$ given $\pi(i), \pi(j)$, and that the conditional distribution
of $U$ given $\pi(i), \pi(j)$ is the same as that of $S$ in \eq
{2.14} for $m=2$ under a relabeling of indices, we can apply the
concentration inequality \eq{5-p1-1} to obtain the following upper
bound on $|R_{2,2}|$.
\begin{eqnarray*}
|R_{2,2}|&\le&\frac{1}{4n} \sum
_{1\le i\ne j\le n} \E(X_{i\pi
(j)}+X_{j\pi(i)}-X_{i\pi(i)}-X_{j\pi(j)})
\\
&&\hphantom{\frac{1}{4n} \sum
_{1\le i\ne j\le n}}{} \times\int_0^{X_{i\pi(j)}+X_{j\pi(i)}-X_{i\pi
(i)}-X_{j\pi(j)}} I\bigl( - (t\vee0)
\le U-(z-X_{i\pi(i)}-X_{j\pi(j)})\\
&&\hphantom{\frac{1}{4n} \sum
_{1\le i\ne j\le n}{} \times\int_0^{X_{i\pi(j)}+X_{j\pi(i)}-X_{i\pi
(i)}-X_{j\pi(j)}} I\bigl(}\le- (t\wedge0)\bigr) \,\mathrm{d}t
\\
&\le&\frac{1}{4n}\sum_{1\le i\ne j\le n}
\E(X_{i\pi
(j)}+X_{j\pi
(i)}-X_{i\pi(i)}-X_{j\pi(j)})
\\
&&\hphantom{\frac{1}{4n}\sum_{1\le i\ne j\le n}}{} \times\int_0^{X_{i\pi(j)}+X_{j\pi(i)}-X_{i\pi
(i)}-X_{j\pi(j)}} \bigl(
c_1 |t|+c_2 \gamma \bigr)\, \mathrm{d}t
\\
&=& \frac{c_1}{8n}\sum_{1\le i\ne j\le n}
\E|X_{i\pi
(j)}+X_{j\pi
(i)}-X_{i\pi(i)}-X_{j\pi(j)}|^3
\\
&&{} +\frac{c_2 \gamma}{4n} \sum_{1\le i\ne j\le n} \E
|X_{i\pi
(j)}+X_{j\pi(i)}-X_{i\pi(i)}-X_{j\pi(j)}|^2
\\
&\le&8c_1 \gamma+c_2 \gamma\biggl(1+
\frac{1}{\sqrt{n}}\biggr). %
\end{eqnarray*}
In the last inequality, we used \eq{2.31} and
\begin{eqnarray*}
\sum_{1\le i\ne j\le n} \E|X_{i\pi(j)}+X_{j\pi(i)}-X_{i\pi
(i)}-X_{j\pi(j)}|^2&=&n(n-1)
\E\bigl(W'-W\bigr)^2
\\
&\le&4n\biggl(1+\frac{1}{\sqrt{n}}\biggr) %
\end{eqnarray*}
which follows from \eq{370}.
For $R_{2,1}$, from the property \eq{1.16} of $f$ with $w=U$,
$u=X_{i\pi(i)}+X_{j\pi(j)}$, $v=X_{i\pi(i)}+X_{j\pi(j)}+t$,
\begin{eqnarray*}
|R_{2,1}|&\le&\frac{1}{4n} \sum
_{1\le i\ne j\le n} \E(X_{i\pi
(j)}+X_{j\pi(i)}-X_{i\pi(i)}-X_{j\pi(j)})
\\
&&\hphantom{\frac{1}{4n} \sum
_{1\le i\ne j\le n}}{} \times\int_0^{X_{i\pi(j)}+X_{j\pi(i)}-X_{i\pi(i)}-X_{j\pi(j)}} \biggl(|U|+
\frac{\sqrt{2\uppi}}{4}\biggr)\bigl(2|X_{i\pi(i)}+X_{j\pi(j)}|+t\bigr)\,\mathrm{d}t
\\
&=&\frac{1}{4n}\sum_{1\le i\ne j\le n} \E\biggl(|U|+
\frac{\sqrt{2\uppi
}}{4}\biggr) \biggl[ (X_{i\pi(j)}+X_{j\pi(i)}-X_{i\pi(i)}-X_{j\pi
(j)})^22|X_{i\pi(i)}+X_{j\pi(j)}|
\\
&&\hphantom{\frac{1}{4n}\sum_{1\le i\ne j\le n} \E\biggl(|U|+
\frac{\sqrt{2\uppi
}}{4}\biggr) \biggl[}{} +\frac{|X_{i\pi(j)}+X_{j\pi(i)}-X_{i\pi(i)}-X_{j\pi
(j)}|^3}{2} \biggr]
\\
&\le&24\biggl(\sqrt{\frac{n}{n-5}+\frac{24n}{(n-5)^2}}+\frac{\sqrt{2\uppi
}}{4}
\biggr)\gamma ,%
\end{eqnarray*}
where we used \eq{ES2}, \eq{2.31} and
\begin{eqnarray*}
&& |X_{i\pi(j)}+X_{j\pi(i)}-X_{i\pi(i)}-X_{j\pi(j)}|^2|X_{i\pi
(i)}+X_{j\pi(j)}|
\\
&&\quad\le\tfrac{16}{3} \bigl( |X_{i\pi(j)}|^3+
|X_{j\pi(i)}|^3\bigr)+\tfrac{32}{3} \bigl(
|X_{i\pi(i)}|^3+ |X_{j\pi(j)}|^3\bigr).
\end{eqnarray*}
Therefore, with
%
\begin{eqnarray}
\label{cthree} c_3&:=&\sqrt{\frac{n}{n-5}+
\frac{24n}{(n-5)^2}}, %
\\
\label{R2} |R_2|&\le&\biggl( 8c_1 +c_2
\biggl(1+\frac{1}{\sqrt{n}}\biggr)+24c_3 +6\sqrt{2\uppi} \biggr)\gamma.
\end{eqnarray}
Next, we bound $R_1$.
\begin{eqnarray*}
R_1&= &\E f'(W) \biggl(1-\frac{(W'-W)^2}{2\lambda}\biggr)
\\
&=&\frac{1}{n^2(n-1)^2}
\\
&&{} \times\sum_{1\le i,j,k,l\le n,\atop i\ne j, k\ne l} \E \biggl(
f'(W) \biggl(1-\frac{(X_{il}+X_{jk}-X_{ik}-X_{jl})^2}{2\lambda}\biggr) \bigl|\\
&& \hphantom{{} \times\sum_{1\le i,j,k,l\le n,\atop i\ne j, k\ne l} \E \biggl(}I=i, J=j ,\pi(i)=k,
\pi(j)=l \biggr)
\\
&=&\frac{1}{n^2(n-1)^2} \sum_{1\le i,j,k,l\le n,\atop i\ne j,
k\ne l} \E \biggl(
f'(W) \biggl(1-\frac{(X_{il}+X_{jk}-X_{ik}-X_{jl})^2}{2\lambda
}\biggr) \bigl| \pi(i)=k, \pi(j)=l
\biggr) %
\end{eqnarray*}
since $\mathbb{X}, \pi$ and $(I,J)$ are independent.
For each choice of $i\ne j, k\ne l$, let $\mathbb{X}^{ijkl}:=\{
X^{ijkl}_{i'j'}\dvt  i', j'\in[n]\}$ be the same as $\mathbb{X}$ except
that $\{X_{ik}, X_{il}, X_{jk}, X_{jl}\}$ has been replaced by an
independent copy $\{X_{ik}', X_{il}', X_{jk}', X_{jl}'\}$. Define
\begin{eqnarray*}
W^{ijkl}=\sum_{i'=1}^n
X_{i'\pi(i')}^{ijkl}.
\end{eqnarray*}
Then,
%
\begin{equation}
\label{0.01} W^{ijkl} \mbox{ is independent of } \{X_{ik},
X_{il}, X_{jk}, X_{jl}\} \mbox{ and } \mathcal{L}
\bigl(W^{ijkl}\bigr)=\mathcal{L}(W).
\end{equation}
Next, we define a new permutation $\pi_{ijkl}$ coupled with $\pi$
such that
\begin{eqnarray*}
\mathcal{L}(\pi_{ijkl})=\mathcal{L}\bigl(\pi | \pi(i)=k, \pi(j)=l\bigr).
\end{eqnarray*}
This coupling has been constructed by Goldstein [11].
Let $\tau_{ij}$ denote the transposition of $i, j$. Define
\begin{equation}\label{1001}
\pi_{ijkl}=
\cases{
\pi &  \quad \mbox{if} $l=\pi(j), k= \pi(i)$,\cr
\pi\cdot \tau_{\pi^{-1}(k), i} &  \quad \mbox{if} $l=\pi(j), k\ne \pi(i)$,\cr
\pi\cdot \tau_{\pi^{-1}(l), j} & \quad \mbox{if} $l\ne \pi(j), k= \pi(i)$,\cr
\pi\cdot \tau_{\pi^{-1}(l), i}\cdot \tau_{\pi^{-1}(k), j}\cdot \tau_{ij} & \quad \mbox{if} $l\ne \pi(j), k\ne \pi(i)$.
}
\end{equation}
Let
\begin{eqnarray*}
W_{ijkl}=\sum_{i'=1}^n
X_{i'\pi_{ijkl}(i')}.
\end{eqnarray*}
Since $W_{ijkl}$ has the conditional distribution of $W$ given $\pi
(i)=k, \pi(j)=l$, and since $\mathbb{X}$ and $\pi$ are independent,
we have
\begin{eqnarray*}
&&\E \biggl( f'(W) \biggl(1-\frac{(X_{il}+X_{jk}-X_{ik}-X_{jl})^2}{2\lambda
}\biggr) \bigl|
\pi(i)=k, \pi(j)=l \biggr)
\\
&&\quad=\E f'(W_{ijkl}) \biggl(1-\frac
{(X_{il}+X_{jk}-X_{ik}-X_{jl})^2}{2\lambda}\biggr).
\end{eqnarray*}
Therefore,
\begin{eqnarray*}
R_1
&= &\E f'(W) \biggl(1-\frac{(W'-W)^2}{2\lambda}\biggr)
\\
&= &\frac{1}{(n(n-1))^2} \sum_{1\le i,j,k,l\le n,\atop i\ne j,
k\ne l} \E\biggl(1-
\frac{(X_{il}+X_{jk}-X_{ik}-X_{jl})^2}{2\lambda}\biggr) {\bigl(} f'(W_{ijkl})
-f'\bigl(W^{ijkl}\bigr) {\bigr)}
\\
&&{} + \frac{1}{(n(n-1))^2} \sum_{1\le i,j,k,l\le n,\atop i\ne
j, k\ne
l} \E
\biggl(1-\frac{(X_{il}+X_{jk}-X_{ik}-X_{jl})^2}{2\lambda}\biggr) f'\bigl(W^{ijkl}\bigr).
\end{eqnarray*}
Define index sets $\mathcal{I}=\{i,j,\pi^{-1}(k), \pi^{-1}(l)\}$ and
$\mathcal{J}=\{k, l, \pi(i), \pi(j)\}$. Then $|\mathcal
{I}|=|\mathcal{J}|\in\{2,3,4\}$. Letting $S=\sum_{i'\notin\mathcal
{I}} X_{i'\pi(i')}$, we can write
\begin{eqnarray*}
W_{ijkl}=S+\sum_{i'\in\mathcal{I}} X_{i'\pi_{ijkl}(i')}
\end{eqnarray*}
and
\begin{eqnarray*}
W^{ijkl}=S+\sum_{i'\in\mathcal{I}} X^{ijkl}_{i'\pi(i')}.
\end{eqnarray*}
Since $S$ is a function depending only on the components of $\mathbb
{X}$ outside the square $\mathcal{I}\times\mathcal{J}$ and $\{\pi
(i)\dvt  i\notin\mathcal{I}\}$,
%
\begin{eqnarray}
\label{3120} &&S \mbox{ is independent of} \biggl
\{X_{il}, X_{jk}, X_{ik}, X_{jl}, \sum
_{i'\in\mathcal{I}} X_{i'\pi_{ijkl}(i')}, \sum
_{i'\in\mathcal{I}} X^{ijkl}_{i'\pi(i')}\biggr\}\nonumber
\\[-8pt]\\[-8pt]
&&\mbox{given } \pi^{-1}(k), \pi^{-1}(l), \pi(i), \pi(j).\nonumber
\end{eqnarray}
The conditional distribution of $S$ given $\pi^{-1}(k), \pi^{-1}(l),
\pi(i), \pi(j)$ is the same as that of $S$ in \eq{2.14} under a
relabeling of indices.
From \eq{ES2}, $\E (|S| |\pi^{-1}(k), \pi^{-1}(l), \pi(i),
\pi(j)  ) \le c_3$ where $c_3$ was defined in \eq{cthree}. From
\eq{0.01}, \eq{1.17} and \eq{370},
\begin{eqnarray*}
&&\frac{1}{(n(n-1))^2} \biggl| \sum_{1\le i,j,k,l\le n,\atop i\ne j,
k\ne l} \E
\biggl(1-\frac{(X_{il}+X_{jk}-X_{ik}-X_{jl})^2}{2\lambda}\biggr) f'\bigl(W^{ijkl}\bigr)\biggr |
\\
&&\quad=\frac{1}{(n(n-1))^2} \biggl| \E f'(W)\sum
_{1\le i,j,k,l\le
n,\atop
i\ne j, k\ne l} \E\biggl(1-\frac{(X_{il}+X_{jk}-X_{ik}-X_{jl})^2}{2\lambda
}\biggr) \biggr|
\\
&&\quad= \biggl| \E f'(W) \E\biggl(1-\frac{(W'-W)^2}{2\lambda}\biggr) \biggr|\le
\frac
{1}{\sqrt{n}}. %
\end{eqnarray*}
Therefore,
\begin{eqnarray*}
 |R_1|&\le& \biggl|\frac{1}{(n(n-1))^2} \sum_{1\le i,j,k,l\le
n,\atop i\ne
j, k\ne l} \E
\biggl(1-\frac{(X_{il}+X_{jk}-X_{ik}-X_{jl})^2}{2\lambda}\biggr) {\bigl(} f'(W_{ijkl})
-f'\bigl(W^{ijkl}\bigr) {\bigr)} \biggr|
\nonumber
\\
&&{} +\frac{1}{\sqrt{n}}
\\
&=& \biggl|\frac{1}{(n(n-1))^2} \sum_{1\le i,j,k,l\le n,\atop
i\ne j,
k\ne l} \E
\biggl(1-\frac{(X_{il}+X_{jk}-X_{ik}-X_{jl})^2}{2\lambda
}\biggr)
\nonumber
\\
&&\hphantom{{\biggl|}\frac{1}{(n(n-1))^2} \sum_{1\le i,j,k,l\le n,\atop
i\ne j,
k\ne l}}{} \times {\biggl(} f'\biggl(S+\sum
_{i'\in\mathcal{I}} X_{i'\pi
_{ijkl}(i')}\biggr) -f'\biggl(S+\sum
_{i'\in\mathcal{I}} X^{ijkl}_{i'\pi(i')}\biggr) {
\biggr)} \biggr|+\frac{1}{\sqrt{n}}
\nonumber
\\
&\le& R_{1,1}+R_{1,2}+\frac{1}{\sqrt{n}}, %
\end{eqnarray*}
where using the Stein equation \eq{1.9},
\begin{eqnarray*}
R_{1,1}&=& \biggl| \frac{1}{(n(n-1))^2} \sum
_{1\le i,j,k,l\le n,\atop
i\ne j, k\ne l} \E\biggl(1-\frac{(X_{il}+X_{jk}-X_{ik}-X_{jl})^2}{2\lambda
}\biggr)
\nonumber
\\
&&\hphantom{\biggl| \frac{1}{(n(n-1))^2} \sum
_{1\le i,j,k,l\le n,\atop
i\ne j, k\ne l}}{} \times {\biggl(} \biggl(S+\sum_{i'\in\mathcal{I}}
X_{i'\pi
_{ijkl}(i')}\biggr)f\biggl(S+\sum_{i'\in\mathcal{I}}
X_{i'\pi_{ijkl}(i')}\biggr)\\
&&\hphantom{\biggl| \frac{1}{(n(n-1))^2} \sum
_{1\le i,j,k,l\le n,\atop
i\ne j, k\ne l}{} \times {\biggl(}}{} -\biggl(S+\sum_{i'\in\mathcal{I}}
X^{ijkl}_{i'\pi(i')}\biggr)f\biggl(S+\sum
_{i'\in
\mathcal{I}} X^{ijkl}_{i'\pi(i')}\biggr) {\biggr)} \biggr| ,
\end{eqnarray*}
and
\begin{eqnarray*}
R_{1,2}&=& \biggl| \frac{1}{(n(n-1))^2} \sum
_{1\le i,j,k,l\le n,\atop
i\ne j, k\ne l} \E\biggl(1-\frac{(X_{il}+X_{jk}-X_{ik}-X_{jl})^2}{2\lambda
}\biggr)
\\
&&\hphantom{| \frac{1}{(n(n-1))^2} \sum
_{1\le i,j,k,l\le n,\atop
i\ne j, k\ne l} }{} \times {\biggl(} I\biggl(S+\sum_{i'\in\mathcal{I}}
X_{i'\pi
_{ijkl}(i')}\le z\biggr)-I\biggl(S+\sum_{i'\in\mathcal{I}}
X^{ijkl}_{i'\pi
(i')}\le z\biggr) {\biggr)} \biggr| . %
\end{eqnarray*}
Applying \eq{1.16}, \eq{3120}, \eq{ES2} and \eq{1.2b}, $R_{1,1}$
can be bounded as follows.
\begin{eqnarray*}
R_{1,1}&\le&\frac{1}{(n(n-1))^2}\\
&&{}\times\sum
_{1\le i,j,k,l\le n,\atop i\ne j,
k\ne l}\E \biggl\{\E \bigl( |S| \bigl|\pi^{-1}(k),
\pi^{-1}(l), \pi (i), \pi(j) \bigr)+\frac{\sqrt{2\uppi}}{4} \biggr\}
\\
&&\hphantom{{}\times\sum
_{1\le i,j,k,l\le n,\atop i\ne j,
k\ne l}}{} \times\E \biggl\{ \biggl(\biggl|\sum_{i'\in\mathcal{I}}
X_{i'\pi
_{ijkl}(i')}\biggr|+\biggl|\sum_{i'\in\mathcal{I}} X^{ijkl}_{i'\pi(i')}\biggr|
\biggr)
\\
&&\hphantom{{}\times\sum
_{1\le i,j,k,l\le n,\atop i\ne j,
k\ne l}{} \times\E \biggl\{}{} \times \biggl| 1-\frac
{(X_{il}+X_{jk}-X_{ik}-X_{jl})^2}{2\lambda} \biggr| \bigl| \pi ^{-1}(k),
\pi^{-1}(l), \pi(i), \pi(j) \biggr\}
\\
&\le&\frac{c_3+\sqrt{2\uppi}/4}{n^2(n-1)^2}\\
&&{}\times\sum_{1\le i,j,k,l\le
n,\atop i\ne j, k\ne l} \max \biggl\{ \E
\biggl( \sum_{i'\in\mathcal
{I}} |X_{i' \pi_{ijkl}(i')}| +\sum
_{i'\in\mathcal{I}} \bigl|X_{i'\pi
(i')}^{ijkl}\bigr| \biggr),
\\
&& \hphantom{{}\times\sum_{1\le i,j,k,l\le
n,\atop i\ne j, k\ne l} \max \biggl\{}\frac{4}{2\lambda}\E \biggl( \sum_{i'\in\mathcal{I}}
|X_{i' \pi_{ijkl}(i')}| +\sum_{i'\in\mathcal{I}} \bigl|X_{i'\pi
(i')}^{ijkl}\bigr|
\biggr) \bigl( X_{il}^2+X_{jk}^2+X_{ik}^2+X_{jl}^2
\bigr) \biggr\}
\\
&\le&\frac{c_3+\sqrt{2\uppi}/4}{n^2(n-1)^2} \sum_{1\le i,j,k,l\le
n,\atop i\ne j, k\ne l} \max \biggl\{ \E
\bigl[ |X_{ik}| + |X_{jl}|+ |X_{\pi^{-1}(k)\pi(i)}
|+|X_{\pi^{-1}(l)\pi
(j)}|
\\
&&\hphantom{\frac{c_3+\sqrt{2\uppi}/4}{n^2(n-1)^2} \sum_{1\le i,j,k,l\le
n,\atop i\ne j, k\ne l} \max \biggl\{ \E
\bigl[}{} +\bigl|X_{i\pi(i)}^{ijkl}\bigr|+\bigl|X_{j\pi(j)}^{ijkl}\bigr|+\bigl|X_{\pi
^{-1}(k)k}^{ijkl}\bigr|+\bigl|X_{\pi^{-1}(l)l}^{ijkl}\bigr|
\bigr],
\\
&&\hphantom{\frac{c_3+\sqrt{2\uppi}/4}{n^2(n-1)^2} \sum_{1\le i,j,k,l\le
n,\atop i\ne j, k\ne l} \max \biggl\{} \frac{2}{\lambda}\E \biggl( \sum_{i'\in\mathcal{I}}
\frac
{4}{3}|X_{i' \pi_{ijkl}(i')}|^3 +\sum
_{i'\in\mathcal{I}}\frac
{4}{3} \bigl|X_{i'\pi(i')}^{ijkl}\bigr|^3\\
&&\hphantom{\frac{c_3+\sqrt{2\uppi}/4}{n^2(n-1)^2} \sum_{1\le i,j,k,l\le
n,\atop i\ne j, k\ne l} \max \biggl\{\frac{2}{\lambda}\E \biggl(}{}+ \frac{16}{3} \bigl( |X_{il}|^3+|X_{jk}|^3+|X_{ik}|^3
+|X_{jl}|^3 \bigr) \biggr) \biggr\}
\\
&\le&\max \Biggl\{ \frac{8(c_3+\sqrt{2\uppi}/4)}{n^2} \sum_{i,k=1}^n
\E|X_{ik}|, 32\biggl(c_3+\frac{\sqrt{2\uppi}}{4}\biggr)
\frac{n-1}{n}\gamma \Biggr\}
\\
&\le&\max \biggl\{ 8 \biggl(c_3+\frac{\sqrt{2\uppi}}{4}\biggr)
\frac{1}{\sqrt {n}}, 32\biggl(c_3+\frac{\sqrt{2\uppi}}{4}\biggr)
\frac{n-1}{n}\gamma \biggr\}, %
\end{eqnarray*}
where we used
%
\begin{equation}
\label{maxstar} \biggl|1- \frac{(X_{il}+X_{jk}-X_{ik}-X_{jl})^2}{2\lambda} \biggr| \le \max \biggl\{1, \frac{(X_{il}+X_{jk}-X_{ik}-X_{jl})^2}{2\lambda}
\biggr\}
\end{equation}
and
\begin{eqnarray*}
\sum_{i,k=1}^n \E|X_{ik}|\leq n
\sqrt{\sum_{i,k=1}^n \E
X_{ik}^2}\leq n^{3/2}.
\end{eqnarray*}
Now we bound $R_{1,2}$.
\begin{eqnarray*}
R_{1,2}&=& \biggl| \frac{1}{(n(n-1))^2} \\
&&{\hphantom{\biggl|}}{}\times\sum
_{1\le i,j,k,l\le n,\atop
i\ne j, k\ne l} \E \biggl\{ \E \biggl[ \biggl(1-\frac
{(X_{il}+X_{jk}-X_{ik}-X_{jl})^2}{2\lambda}
\biggr)
\\
&&\hphantom{\biggl|{}\times\sum
_{1\le i,j,k,l\le n,\atop
i\ne j, k\ne l} \E \biggl\{ \E \biggl[}{} \times {\biggl(} I\biggl(S+\sum_{i'\in\mathcal{I}}
X_{i'\pi
_{ijkl}(i')}\le z\biggr)-I\biggl(S+\sum_{i'\in\mathcal{I}}
X^{ijkl}_{i'\pi
(i')}\le z\biggr) {\biggr)} \bigl|\\
&&\hphantom{|\sum
_{1\le i,j,k,l\le n,\atop
i\ne j, k\ne l} \E \biggl\{ \E \biggl[} \pi^{-1}(k),
\pi^{-1}(l), \pi(i), \pi (j) \biggr] \biggr\} \biggr|
\\
&\le&\frac{1}{(n(n-1))^2} \\
&&{}\times\sum_{1\le i,j,k,l\le n,\atop i\ne j, k\ne
l} \E \biggl\{ \E
\biggl[ \biggl|1-\frac
{(X_{il}+X_{jk}-X_{ik}-X_{jl})^2}{2\lambda}\biggr|
\nonumber
\\
& &\hphantom{{}\times\sum_{1\le i,j,k,l\le n,\atop i\ne j, k\ne
l} \E \biggl\{ \E
\biggl[}{}\times I\biggl(z-\max\biggl\{\sum_{i'\in\mathcal{I}}
X_{i'\pi
_{ijkl}(i')},\sum_{i'\in\mathcal{I}} X^{ijkl}_{i'\pi(i')}
\biggr\} \\
&&\hphantom{{}\times\sum_{1\le i,j,k,l\le n,\atop i\ne j, k\ne
l} \E \biggl\{ \E
\biggl[{}\times I\biggl(}{} \le S\le z-\min\biggl\{\sum_{i'\in\mathcal{I}}
X_{i'\pi_{ijkl}(i')} , \sum_{i'\in\mathcal{I}} X^{ijkl}_{i'\pi(i')}
\biggr\}\biggr)
\bigl |\\
&&\hphantom{{}\times\sum_{1\le i,j,k,l\le n,\atop i\ne j, k\ne
l} \E \biggl\{ \E
\biggl[} \pi^{-1}(k), \pi^{-1}(l), \pi(i), \pi(j) \biggr] \biggr
\}. %
\end{eqnarray*}
Recall \eq{3120} and the fact that the conditional distribution of $S$
given $\pi^{-1}(k)$, $\pi^{-1}(l)$, $\pi(i)$, $\pi(j)$ is the same
as that of $S$ in \eq{2.14} under a relabeling of indices. We can
therefore apply the concentration inequality \eq{5-p1-1} to obtain the
following upper bound on $R_{1,2}$.
\begin{eqnarray*}
R_{1,2}&\le&\frac{1}{(n(n-1))^2} \sum
_{1\le i,j,k,l\le n,\atop i\ne
j, k\ne l} \E\biggl|1-\frac{(X_{il}+X_{jk}-X_{ik}-X_{jl})^2}{2\lambda}\biggr|
\\
&&\hphantom{\frac{1}{(n(n-1))^2} \sum
_{1\le i,j,k,l\le n,\atop i\ne
j, k\ne l}}{} \times {\biggl\{} c_1 \biggl(\biggl|\sum
_{i'\in\mathcal{I}} X_{i'\pi
_{ijkl}(i')}\biggr|+\biggl|\sum_{i'\in\mathcal{I}}
X^{ijkl}_{i'\pi(i')}\biggr|\biggr)+c_2 \gamma {\biggr\}}
\\
&\le& c_1 \max {\biggl\{}\frac{8}{\sqrt{n}}, 32\frac{n-1}{n}
\gamma {\biggr\}}+c_2 \biggl(1+\frac{1}{\sqrt{n}}\biggr) \gamma,
\end{eqnarray*}
where we used \eq{maxstar} and \eq{370}.
Therefore,
%
\begin{equation}
\label{R1} |R_1|\le\biggl(c_1
+c_3+\frac{\sqrt{2\uppi}}{4}\biggr)\max {\biggl\{} \frac
{8}{\sqrt{n}}, 32
\frac{n-1}{n}\gamma {\biggr\}} +c_2 \biggl(1+\frac
{1}{\sqrt{n}}
\biggr) \gamma+\frac{1}{\sqrt{n}}. %
\end{equation}
Summing \eq{R1}, \eq{R2}, \eq{R3} yields an upper bound on $\sup_{z\in\mathbb{R}}|P(W\le z)-\Phi(z)|$ as
%
\begin{eqnarray}
\label{3131} &&\sup_{z\in\mathbb{R}}\bigl|P(W\le z)-\Phi(z)\bigr|\nonumber
\\[-8pt]\\[-8pt]
&&\quad\le \biggl( 40c_1+2\biggl(1+\frac{1}{\sqrt{n}}
\biggr)c_2+14\sqrt{2\pi }+56c_3+2\biggl(
\frac{n}{n-1}\biggr)^{3/2} \biggr) \gamma,\nonumber %
\end{eqnarray}
where we used \eq{301}. Recall $c_0=451$ and $n\ge203\,000$. Using
$c_0=451$ and $n=203\,000$, the upper bound in \eq{3131} is calculated
to be smaller than $451\gamma$. Since $c_1, c_2$ and $c_3$ decrease as
$n$ increases, \eq{t1-1} holds for $n\ge203\,000$. This completes the
proof of Theorem~\ref{t1}.

\begin{rem}
Radoslaw Adamczak has brought to our attention an inconsistency between (4.3) and (4.6)
in Ho and Chen \cite{HoCh78}.  This error can be corrected by defining $\rho$ given $(I,K,L,M)$ in
(4.6) as $\pi_{IKLM}$ in (\ref{1001}) in this paper. This correction will not affect the rest of Ho and Chen~\cite{HoCh78}.
\end{rem}

\section*{Acknowledgements}

This work is based on part of the Ph.D. thesis of the second author.
The second author is thankful to the first author for his guidance and
helpful discussions. Both the authors would like to thank the referee
for their very detailed and helpful comments. We would also like to thank Radoslaw Adamczak for
bringing to our attention an error in Ho and Chen \cite{HoCh78}.
This work is partially
supported by Grant C-389-000-010-101 and Grant C-389-000-012-101 at the
National University of Singapore.




\printhistory

\end{document}